\documentclass[reqno,10pt]{amsart}
\numberwithin{equation}{section}

\usepackage{amssymb,amsmath,latexsym}
\usepackage{amsmath, amsthm, amscd, amsfonts, amssymb, graphicx, color}
\usepackage[colorlinks=true,linkcolor=red,citecolor=blue]{hyperref}
\usepackage{marvosym}
\usepackage{epsfig, epstopdf}
\usepackage{cite}
\usepackage{float}
\title[   ]{ A generalized strong convergence algorithm in the presence of errors for  variational inequality problems in Hilbert spaces}
 \author[M. Ghadampour, D. O'Regan,    E. Soori  and R. P. Agarwal ]{Mostafa Ghadampour$^{1}$,  Donal O'Regan $^{ 2}$ ,    $\,\,$ Ebrahim Soori$^{*,3}$ $\,\,$  and Ravi P. Agarwal $^{*,4}$ }
 \thanks{ \!\! \!\! \!\! \!\!* Corresponding author  
  \\ E-mail addresses: m.ghadampour@gmail.com( m. ghadampour); sori.e@lu.ac.ir (E. Soori);  donal.oregan@nuigalway.ie (D. O'Regan);  agarwal@tamuk.edu (R. P. Agarwal). }

\theoremstyle{plain}
\newtheorem{lem}{\textbf{Lemma}}[section]
\newtheorem{thm}[lem]{\textbf{Theorem}}

\newtheorem{ex}[lem]{\textbf{Example}}
\newtheorem{op}{\textbf{Open problem}}
\theoremstyle{definition}

\theoremstyle{definition}
\theoremstyle{remark}

\begin{document}
\maketitle

\begin{center}
\begin{normalsize}
   $^{1, 3}$ Department  of Mathematics, Lorestan University, Lorestan, Khoramabad, Iran,\\
  $^{2}$ School of Mathematics, Statistics, National University of Ireland, Galway, Ireland,\\
    $^{4}$ Department  of Mathematics, Texas A $\&$ M University Kingsville, Kingsville, USA.

 \end{normalsize}
  \end{center}
\begin{abstract}
\begin{normalsize}
In this paper, we study the strong convergence of an algorithm to solve the variational inequality problem which extends a recent paper (Thong et al, Numerical Algorithms. 78, 1045-1060 (2018)). We  reduce and refine some of their algorithm conditions and  we prove the convergence of the algorithm in the presence of some computational errors. Then using MATLAB software, the result will be illustrated with some numerical examples. Also, we compare our algorithm with some other well known  algorithms.
\end{normalsize}
\end{abstract}
\begin{normalsize}
   \end{normalsize}

\section{ Introduction}
Let $H$ be a real Hilbert space with the inner product $\langle., .\rangle$ and the norm $\|.\|$, and $C$ be a nonempty, closed and convex subset of $H$. The variational inequality $(VI)$ is to find a point $x \in C$ such that
\begin{equation}\label{vi}
  \langle Ax, y-x\rangle \geq 0, \;\; \forall y\in C,
\end{equation}
where $A$ is a mapping of C into $H$. The solutions set
of \eqref{vi} is denoted by $VI(A, C)$.
Variational inequalies arise in  the study of  network equilibriums,  optimization problems,  saddle point problem, Nash equilibrium  problems in noncooperative games etc.; sSee for example \cite{ags,th,tvc,bcpp,dm1,dm2,gr1,ks,ks2,ms,wxw} and the references therein.

A new algorithm was proposed  by Korpelevich \cite{ko}  for solving  problem $(VI)$ in the Euclidean space which is known as the extragradient method. Let  $x_1$ be an arbitrarily element in $H$ and consider
     \begin{equation}\label{algo0}
             \begin{cases}
                y_n=P_C(x_n-\lambda Ax_n), \\
                x_{n+1}=P_C(x_n-\lambda Ay_n).
             \end{cases}
          \end{equation}
where $\lambda$ is a number in $(0, 1)$, $P_C$ is the Euclidean least distance  projection of $H$ onto $C$ and
$A : C\rightarrow H$ is a monotone operator.  The next algorithm \eqref{algo1} was introduced by Tseng \cite{tp} and  applying the modified forward-backward(F-B) method  is a good alternative to the extragradient method (TEGM):
\begin{equation}\label{algo1}
     \begin{cases}
       y_{n}=P_C(x_n-\lambda Ax_n),  \\
       x_{n+1}=P_X(y_n-\lambda(Ay_n-Ax_n)),
     \end{cases}
   \end{equation}
where $X = C$ and $X=H$ if $A$ is Lipschitz continuous.
The following algorithm \eqref{algo3} was proposed by Shehu and Iyiola \cite{si} which a viscosity type subgradient extragradient
 method (VSEGM):
           \begin{equation}\label{algo3}
              \begin{cases}
                y_n=P_C(x_n-\lambda_n Ax_n), \\
                T_n=\{z \in H : \langle x_n - \lambda_n Ax_n - y_n, z-y_n \rangle \leq 0\},\\
                z_n= P_{T_n}( x_n - \lambda_n Ay_n),\\
                x_{n+1}=\alpha_nf(x_n)+(1-\alpha_n) z_n,
              \end{cases}
           \end{equation}
where the operator $A$ is monotone and Lipschitz continuous, $f$ is a strict
contraction mapping, $ l\in (0,1),  \;\mu \in (0, 1)$ and $\lambda_n= l^{m_n}$ where $m_n$ is the smallest nonnegative integer $m$ such that
           \begin{align*}
             \lambda\| Ax_n - Ay_n \| \leq \mu \|r_{l^{m_n}} (x_n) \|,
           \end{align*}
            where $r_{l^{m_n}}(x)= x- P_C(x - l^{m_n} A(x))$ for all $x\in C$.
             Recently, the sequence produced by the following algorithm was  introduced by Thong and Hieu \cite{th} based on Tseng's method (THEGM):
           \begin{equation}\label{algo}
              \begin{cases}
                y_n=P_C(x_n-\lambda_n Ax_n), \\
                z_n= y_n - \lambda_n(Ay_n -Ax_n),\\
                x_{n+1}=\alpha_nf(x_n)+(1-\alpha_n) z_n,
              \end{cases}
           \end{equation}
where the operator $A$ is monotone and Lipschitz continuous, $\gamma > 0, \; l\in (0,1),  \;\mu \in (0, 1)$ and $\lambda_n$ is chosen to be the largest $\lambda \in \{ \gamma, \gamma l, \gamma l^2, ...\}$ satisfying
           \begin{align}\label{1}
             \lambda\| Ax_n - Ay_n \| \leq \mu \| x_n - y_n \|.
           \end{align}

In this paper  substituting a sequence   $\{\beta_n\}\subset(0, 1)$ of coefficients  instead of the sequence $\{1-\alpha_n\}$ in the  algorithm \eqref{algo},   we extend algorithm \eqref{algo}. Moreover,  condition \eqref{1} will be removed just by a slight change in the coefficients  $\{\lambda_n
 \}$. Also a sequence of  computational errors in our algorithm is considered. The strong convergence of the proposed algorithm to a point of the variational inequality  $VI(C, A)$ will be proved under the presence of computational errors. Finally, some examples will be presented which will examine the  convergence of the proposed algorithm in different situations.
\section{preliminaries}
In this section,  some  basic concepts are presented.

Let $H$ be a real Hilbert space with the inner product $\langle .,.\rangle$	 and norm $\|.\|$ and suppose that $C$ is
a nonempty closed convex subset of $H$ and $A:C \rightarrow H$ is an operator.  The operator $A$ is said to be \\
 (i) monotone  if
         \begin{align*}
           \langle Ax - Ay, x - y\rangle\geq 0,\;\; \forall x, y \in C;
         \end{align*}
(ii) $L$-Lipchitz continuous if there exist $L>0$ such that
          \begin{align*}
            \| Ax-Ay\|\leq L\| x-y\|,\;\;\forall x,y\in C.
          \end{align*}
For the main results of this paper we need the following useful lemmas.
\begin{lem}\label{2.1}
Let $H$ be a real Hilbert space. Then, we have the following well-known
results:
          \begin{align*}
            \| x+y\|^2 &=\| x\|^2+2\langle x,y\rangle+\| y\|^2, \;\;\forall x,y\in H. \\
            \| x+y\|^2 &\leq\| x\|^2+2\langle y,x+y\rangle, \;\;\forall x,y\in H.
          \end{align*}
\end{lem}
\begin{lem}\label{2.3}
($Xu$, \cite{xu})  Let $\{a_n\}$ be a sequence of nonnegative real numbers satisfying the following relation:
            \begin{align*}
              a_{n+1} \leq(1-\alpha_n)a_n+\alpha_n\sigma_n+\gamma_n, \;\;n\geq1,
             \end{align*}
where
\begin{enumerate}
  \item [(a)]  $\{\alpha_n\}\subset[0,1], \sum_{n=1}^{\infty}\alpha_n=\infty$,\\
  \item [(b)]  $ \limsup\sigma_n\leq0$,\\
  \item [(c)]  $\gamma_n\geq0 (n\geq1), \sum_{n=1}^{\infty}\gamma_n<\infty$.
\end{enumerate}
  Then, $a_n\rightarrow0 \;as\; n\rightarrow\infty.$
\end{lem}
\begin{lem}\label{pc}
 \cite{gr} Let $C$ be a closed and convex subset in a real Hilbert space $H$. Then
  $z=P_Cx$ if and only if $\langle x-z, y-z\rangle\leq 0 \;\;\forall y\in C$.
\end{lem}
\begin{lem}\label{2.4}
  \cite{mai} Let $\{a_n\}$ be a sequence of nonnegative real numbers such that
there exists a subsequence $\{a_{n_j}\}$ of $\{a_n\}$ such that $a_{n_j}<a_{n_j+1}$ for all $j\in \mathbb{N}$. Then
there exists a nondecreasing sequence $\{m_k\}$ of $\mathbb{N}$ such that $\lim_{k\rightarrow\infty} m_k=\infty $ and the
following properties are satisfied by all (sufficiently large) number $k\in\mathbb{N}$:
              \begin{align*}
                a_{m_k}\leq  a_{m_k+1}\;\; and \;\;   a_k\leq  a_{m_k+1}
              \end{align*}
In fact, $m_k$ is the largest number $n$ in the set $\{1, 2, ..., k\}$ such that $a_n < a_{n+1}$.
\end{lem}

\begin{lem}\label{2.5}
\cite{th}  Let $\{x_n\}$ be a sequence generated by algorithm \eqref{algo1}.  Then
         \begin{align*}
           \| x_{n+1}-p\|^2\leq\| x_n-p\|^2-(1-\mu^2)\| x_n-y_n\|^2\;\;\forall\;p\in VI(C,A)
        \end{align*}
\end{lem}
  \section{Main results}
Ww immediately present our main result.
\begin{thm}\label{main1}
 Let $C$ be a nonempty closed convex subset of a real Hilbert space $H$, $A$ be a $L$-Lipschitz continuous mapping on $C$ and $\lambda\in(0,1)$ such that $\lambda L<1$. Suppose that $f:H \rightarrow H$ is a contraction mapping with a constant $\rho \in [0,1) $.  Let $\{e_n\}\subseteq H$ be a sequence of computational errors, $x_{0}\in H$ be arbitrary and $\{x_n\}$, $\{y_n\}$ and $\{z_n\}$  be the sequences generated by
   \begin{equation}\label{algo2}
     \begin{cases}
       y_{n}=P_C(x_n-\lambda Ax_n),  \\
       z_n=y_n-\lambda(Ay_n-Ax_n),\\
       x_{n+1}=\alpha_nf(x_n)+\beta_n z_n+e_n,
    \end{cases}
   \end{equation}
where $\{\alpha_n\}$ and $\{\beta_n\}$ are real sequences in $[0,1]$ such that $\alpha_n+\beta_n\leq 1$ for each $n\geq1$.
 Also assume the following conditions:
 \begin{enumerate}
   \item [(a)] $\Sigma_{n=1}^\infty\alpha_n\beta_n=\infty,$\\
   \item [(b)] $\Sigma_{n=1}^\infty(1-\alpha_n-\beta_n)<\infty,$\\
   \item [(c)] $\displaystyle\lim_{n\rightarrow\infty}\frac{(1-\alpha_n-\beta_n)}{\alpha_n}
 =\displaystyle\lim_{n\rightarrow\infty}\alpha_n=0$,\\
    \item [(d)] $\Sigma_{n=1}^\infty \| e_n\|<\infty$.
 \end{enumerate}
   Then
    \begin{enumerate}
      \item [(i)] $VI(C,A)\neq\emptyset$ if and only if $\{x_n\}$ is bounded and $\displaystyle\liminf_{n\rightarrow\infty}\| x_n-y_n\|=0$.\\
      Suppose $VI(C,A)\neq\emptyset$. Then
      \item [(ii)]  if $\displaystyle\lim_{n\rightarrow\infty}\frac{\| e_n\|}{\alpha_n}=0$, then $\{x_n\}$ converges strongly to $q=P_{VI(C,A)}\circ f(q)$,
where $P_{VI(C,A)}\circ f: H \rightarrow VI(C, A)$ is the mapping defined by $P_{VI(C,A)}\circ f(x)=P_{VI(C,A)}(f(x))$
for each $x\in H$.
    \end{enumerate}
\end{thm}
\begin{proof}
(i) Assume that $\{x_n\}$ is a bounded sequence and  $\displaystyle\liminf_{n\rightarrow\infty}\| x_n-y_n\|=0$. Then there exists a subsequence $\{n_i\}\subset\mathbb{N}$ such that $\| x_{n_i}-y_{n_i}\|\rightarrow 0$ when $i\rightarrow\infty$. Note $\{x_{n_i}\}$ is a bounded sequence. Hence there exists a subsequence $\{x_{n_{i_k}}\}$ of $\{x_{n_i}\}$ such that $\{x_{n_{i_k}}\}$ converges weakly to some $x\in C$. Now noting that since $y_{n_{i_k}}=P_C(x_{n_{i_k}}-\lambda Ax_{n_{i_k}})$, for all $z\in C$ we have
          \begin{align*}
            0\geq & \langle x_{n_{i_k}}-\lambda Ax_{n_{i_k}}-y_{n_{i_k}},z-y_{n_{i_k}}\rangle  \\
             = & \langle  x_{n_{i_k}}-y_{n_{i_k}},z-y_{n_{i_k}}\rangle-\lambda\langle Ax_{n_{i_k}},z-y_{n_{i_k}}\rangle \\
             = & \langle  x_{n_{i_k}}-y_{n_{i_k}},z-y_{n_{i_k}}\rangle-\lambda\langle Ax_{n_{i_k}},z-x_{n_{i_k}}\rangle-\lambda\langle Ax_{n_{i_k}},x_{n_{i_k}}-y_{n_{i_k}}\rangle \\
             = & \langle  x_{n_{i_k}}-y_{n_{i_k}},z-y_{n_{i_k}}\rangle-\lambda\langle Ax_{n_{i_k}}-Az,z-x_{n_{i_k}}\rangle-\lambda\langle Az,z-x_{n_{i_k}}\rangle\\
             - & \lambda\langle Ax_{n_{i_k}},x_{n_{i_k}}-y_{n_{i_k}}\rangle\\
             \geq & \langle  x_{n_{i_k}}-y_{n_{i_k}},z-y_{n_{i_k}}\rangle-\lambda\langle Az,z-x_{n_{i_k}}\rangle- \lambda\langle Ax_{n_{i_k}},x_{n_{i_k}}-y_{n_{i_k}}\rangle.
             \end{align*}
Now we have,
$$-\lambda\langle Az,z-x_{n_{i_k}}\rangle\leq  \langle  x_{n_{i_k}}-y_{n_{i_k}},y_{n_{i_k}}-z\rangle + \lambda\langle Ax_{n_{i_k}},x_{n_{i_k}}-y_{n_{i_k}}\rangle.$$
Therefore,
$$-\lambda\langle Az,z-x_{n_{i_k}}\rangle\leq  \|  y_{n_{i_k}}-x_{n_{i_k}}\|\| y_{n_{i_k}}-z\|+ \lambda\| Ax_{n_{i_k}}\|\| x_{n_{i_k}}-y_{n_{i_k}}\|.$$
From  $\displaystyle\lim_{k\rightarrow\infty}\| x_{n_{i_k}}-y_{n_{i_k}}\|=0$, we have
 $-\lambda\langle Az, z-x\rangle\leq0$ for all $z\in C$. Now let $y\in C$ and $0<t<1$, and from the convexity $C$ we have
$y_t=[ty+(1-t)x]\in C$. Therefore
            \begin{align*}
              0\leq \langle Ay_t, y_t-x\rangle= \langle Ay_t, ty-tx\rangle=t\langle Ay_t, y-x\rangle.
            \end{align*}
Since $0<t<1$ then $\langle Ay_t, y-x\rangle\geq0$ for all $y\in C$. Because the mapping $A$ and multiplication are continuous,  if $t\rightarrow0$ then we have $\langle Ax, y-x\rangle\geq0$ for all $y\in C$, i.e; $x\in VI(C,A)$.\\

For the converse fix $p\in VI(C,A)$. Using Lemma \ref{2.5} we have
\begin{align*}\label{zn}
  \| z_n-p\|^2\leq\| x_n-p\|^2-(1-(\lambda L)^2)\| x_n-y_n\|^2.
\end{align*}
  Therefore,
\begin{equation}\label{znxn}
   \| z_n-p\|\leq\| x_n-p\|.
\end{equation}
Using the above inequality  we have
    \begin{alignat*}{1}
      \| x_{n+1}-p\|= &\|\alpha_n f(x_n)+\beta_n z_n+e_n-p\| \\
      = & \| \alpha_n(f(x_n)-p)+\beta_n(z_n-p)-(1-\alpha_n-\beta_n)p+e_n\|\\
      \leq&  \alpha_n\| f(x_n)-p\|+\beta_n\| z_n-p\|+(1-\alpha_n-\beta_n)\| p\|+\| e_n\| \\
      \leq & \alpha_n\| f(x_n)-f(p)\|+\alpha_n\| f(p)-p\|+\beta_n\| x_n-p\|\\
      +&(1-\alpha_n-\beta_n)\| p\|+\| e_n\| \\
      \leq &  \alpha_n\rho\| x_n-p\|+\alpha_n\| f(p)-p\|+\beta_n\| x_n-p\|\\
      +&(1-\alpha_n-\beta_n)\| p\| +\| e_n\| \\
      = & (\alpha_n\rho+\beta_n)\| x_n-p\|+\alpha_n(1-\rho)\frac{\| f(p)-p\|}{1-\rho}\\
      +&(1-\alpha_n-\beta_n)\| p\|+\| e_n\| \\
      \leq & \max\{\| x_n-p\|, \frac{\| f(p)-p\|}{1-\rho}, \| p\|\}+\| e_n\|,
       \end{alignat*}
so the sequence $\{x_n\}$ is bounded. \\

(ii) Let $p\in VI(C,A)$, by part (i) we have that $\{x_n\}$ is bounded. Then $\{f(x_n)\}$, $\{y_n\}$ and $\{z_n\}$ are bounded.
Now  using Lemma \ref{2.5} we have
 \begin{equation}\label{zn<xn}
             \| z_n-p\|^2\leq\| x_n-p\|^2-(1-(\lambda L)^2)\| x_n-y_n\|^2\;\;\forall\;p\in VI(C,A),
          \end{equation}
(note that our  sequence $\{z_n\}$ in algorithm \eqref{algo2}  replaces $\{x_{n+1}\}$ in Lemma \ref{2.5}).\\
 Next, by the convexity of $\| .\|^2$ and the relation $\| x+y\|^2=\| x\|^2+2\langle x,y\rangle+\| y\|^2$, we conclude that,
         \begin{align*}
           \| x_{n+1}-p\|^2= &\|\alpha_n f(x_n)+\beta_n z_n+e_n-p\|^2 \\
            = & \| \alpha_n(f(x_n)+e_n-p)+\beta_n(z_n+e_n-p)+(1-\alpha_n-\beta_n)(e_n-p)\|^2 \\
            \leq & \alpha_n\| f(x_n)+e_n-p\|^2+\beta_n\| z_n+e_n-p\|^2\\
            +&(1-\alpha_n-\beta_n)\| e_n-p\|^2\\
            \leq & \alpha_n\| f(x_n)-p\|^2+(1-\alpha_n)\| z_n-p\|^2\\
            +&(1-\alpha_n-\beta_n))\| p\|^2+(\alpha_n+\beta_n+(1-\alpha_n-\beta_n))\| e_n\|^2 \\
            +&2\langle\alpha_n f(x_n)-\alpha_np+\beta_n z_n-\beta_np-(1-\alpha_n-\beta_n)p, e_n\rangle\\
            \leq & \alpha_n\| f(x_n)-p\|^2+(1-\alpha_n)\| x_n-p\|^2\\
            -&(1-\alpha_n)(1-(\lambda L)^2)\| x_n-y_n\|^2+(1-\alpha_n-\beta_n)\| p\|^2+\| e_n\|^2\\
            +&2\langle \alpha_n f(x_n)+\beta_n z_n-p, e_n\rangle \;\;\;\;\;\;\;\;\;\;\;\;\;\;\;\;\;\;\;\;\;\;\;\;\;\;\;\;\;\; (by\;\eqref{zn<xn}) \\
            \leq & \alpha_n\| f(x_n)-p\|^2+\| x_n-p\|^2\\
            -&(1-\alpha_n)(1-(\lambda L)^2)\| x_n-y_n\|^2
            +(1-\alpha_n-\beta_n)\| p\|^2+\| e_n\|^2\\
            +&2\| \alpha_n f(x_n)+\beta_n z_n-p\|\| e_n\|.
         \end{align*}
Therefore
          \begin{align}\label{xn-yn1}
            (1-\alpha_n)(1-(\lambda L)^2)\| x_n-y_n\|^2\leq&\| x_n-p\|^2 -\| x_{n+1}-p\|^2\nonumber\\
           +&\alpha_n\| f(x_n)-p\|^2+(1-\alpha_n-\beta_n)\| p\|^2\nonumber\\
           +&\| e_n\|^2+2\| \alpha_n f(x_n)+\beta_n z_n-p\|\| e_n\|.
          \end{align}
Also we have,
         \begin{align}\label{xn+1-xn}
          \| x_{n+1}-x_n\|= &\|\alpha_n f(x_n)+\beta_n z_n+e_n-x_n\|\nonumber \\
          = & \| \alpha_n(f(x_n)-x_n)+\beta_n(z_n-x_n)-(1-\alpha_n-\beta_n)x_n+e_n\|\nonumber \\
          \leq& \alpha_n\| f(x_n)-x_n\|+\beta_n\| z_n-x_n\|\nonumber\\
          +&(1-\alpha_n-\beta_n)\| x_n\|+\| e_n\|\nonumber\\
          = & \alpha_n\| f(x_n)-x_n\|+\beta_n\| y_n-\lambda(Ay_n-Ax_n)-x_n\|\nonumber\\
          +&(1-\alpha_n-\beta_n)\| x_n\|+\| e_n\|\nonumber\\
          \leq &  \alpha_n\| f(x_n)-x_n\|+\beta_n\| y_n-x_n\|+\beta_n\lambda L\| y_n-x_n\|\nonumber\\
          +&(1-\alpha_n-\beta_n)\| x_n\|+\| e_n\|.
         \end{align}
Note that $P_{VI(C,A)}\circ f$ is a contraction mapping. Then by Banach contraction principle, there exists a unique element $q \in  H$ such that $q = P_{VI(C,A)}\circ f(q)$. 
Now we claim that $\{x_n\} $  converges strongly to  $q=P_{VI(C,A)}\circ f(q)$. For proved this claim, we will show that $\| x_n-q\|^2\rightarrow 0$ by considering two
possible cases on the sequence $\| x_n-q\|^2$.\\
\textbf{Case1.} Suppose there exists  some $n_0\in\mathbb{N}$ such that $\| x_{n+1}-q\|^2\leq\| x_n-q\|^2$ for all $n\geq n_0$. Then $\displaystyle\lim_{n\rightarrow\infty}\| x_n-q\|$  exists. From  \eqref{xn-yn1}  and our assumptions $\displaystyle\lim_{n\rightarrow\infty}\| x_n-y_n\|=0$, hence from \eqref{xn+1-xn}  we obtain that $\{x_n\}$ is a Cauchy sequence in the Hilbert space $H$, therefore $\{x_n\}$ is convergent. Now we show that $\{x_n\} $  converges strongly to  $q$.
Note that,
         \begin{align*}
            \| x_{n+1}-q\|^2= &\|\alpha_n f(x_n)+\beta_n z_n+e_n-q\|^2 \\
            = & \| \alpha_n(f(x_n)+e_n-q)+\beta_n(z_n+e_n-q)+(1-\alpha_n-\beta_n)(e_n-q)\|^2 \\
            \leq &\beta_n^2 \| z_n+e_n-q\|^2\\
             +2&\langle\alpha_n(f(x_n)+e_n-q)+(1-\alpha_n-\beta_n)(e_n-q),x_{n+1}-q \rangle\;\;\;\;(by\;Lemma \;\ref{2.1}) \\
             = &\beta_n^2 \| z_n+e_n-q\|^2+2\alpha_n\langle(f(x_n)+e_n-q),x_{n+1}-q \rangle \\
             +&2(1-\alpha_n-\beta_n)\langle e_n-q, x_{n+1}-q \rangle\\
            = &\beta_n^2 \| z_n+e_n-q\|^2+2\alpha_n\langle(f(x_n)-q),x_{n+1}-q \rangle\\
              +&2\alpha_n\langle e_n,x_{n+1}-q \rangle+2(1-\alpha_n-\beta_n)\langle e_n-q, x_{n+1}-q \rangle\\
             \leq &\beta_n^2 \| z_n+e_n-q\|^2+2\alpha_n\rho\| x_n-q\|\| x_{n+1}-q \|\\
             +2&\alpha_n\langle f(q)-q,x_{n+1}-q \rangle+ 2\alpha_n\langle e_n,x_{n+1}-q \rangle\\
             +&2(1-\alpha_n-\beta_n)\langle e_n-q, x_{n+1}-q \rangle\\
             \leq &\beta_n^2 \| z_n-q\|^2+2\beta_n^2\langle e_n, z_n+e_n-q\rangle+2\alpha_n\rho\| x_n-q\|^2\;\;\;\;( by\;Lemma\; \ref{2.1})\\
             +2&\alpha_n\langle f(q)-q,x_{n+1}-q \rangle+ 2\alpha_n\langle e_n,x_{n+1}-q \rangle\\
             +&2(1-\alpha_n-\beta_n)\langle e_n-q, x_{n+1}-q \rangle\\
             \leq &(1-\alpha_n)^2 \| x_n-q\|^2+2\alpha_n\rho\| x_n-q\|^2\\
             +2&\alpha_n\langle f(q)-q,x_{n+1}-q \rangle+2\beta_n^2\langle e_n, z_n+e_n-q\rangle\\
             +&2\alpha_n\langle e_n,x_{n+1}-q \rangle+2(1-\alpha_n-\beta_n)\langle e_n-q, x_{n+1}-q \rangle\\
             =&(1-2\alpha_n(1-\rho))\| x_n-q\|^2\\
             +&2\alpha_n(1-\rho)\bigg[\frac{\alpha_n\| x_n-q\|^2}{2(1-\rho)}
             +\frac{\langle f(q)-q,x_{n+1}-q\rangle}{1-\rho}\bigg]\\
             +&2\beta_n^2\langle e_n, z_n+e_n-q\rangle+ 2\alpha_n\langle e_n,x_{n+1}-q \rangle\\
             +&2(1-\alpha_n-\beta_n)\langle e_n-q, x_{n+1}-q \rangle
           \end{align*}
          \begin{align}\label{en}
             \leq&(1-2\alpha_n(1-\rho))\| x_n-q\|^2\nonumber\\
             +&2\alpha_n(1-\rho)\bigg[\frac{\alpha_n\| x_n-q\|^2}{2(1-\rho)}
             +\frac{\langle f(q)-q,x_{n+1}-q\rangle}{1-\rho}\bigg]\nonumber\\
             +&2\beta_n^2\| e_n\|\| z_n+e_n-q\|+ 2\alpha_n\| e_n\|\| x_{n+1}-q \|\nonumber\\
             +&2(1-\alpha_n-\beta_n)\| e_n-q\|\| x_{n+1}-q \|\nonumber\\
             =&(1-2\alpha_n(1-\rho))\| x_n-q\|^2\nonumber\\
             +&2\alpha_n(1-\rho)\bigg[\frac{\alpha_n\| x_n-q\|^2}{2(1-\rho)}
             +\frac{\langle f(q)-q,x_{n+1}-q\rangle}{1-\rho}\\
             +&\frac{\| e_n\|}{\alpha_n(1-\rho)}\beta_n^2\| z_n+e_n-q\|+\frac{\| e_n\|\| x_{n+1}-q \|}{1-\rho}+\frac{(1-\alpha_n-\beta_n)\| e_n-q\|\| x_{n+1}-q \|}{\alpha_n (1-\rho)}\bigg].
        \end{align}
Since $\{x_n\}$ is strongly convergent, so $x_n\rightarrow z$ for some $z\in C$, then we conclude that $x_n\rightharpoonup z$. Also as in the proof of part(i) we conclude $z\in VI(C, A)$.
Therefore from Lemma \ref{pc}
           \begin{align}\label{f(q)-q}
            \displaystyle\limsup_{n\rightarrow\infty}\langle f(q)-q,x_{n+1}-q\rangle=
            \langle f(q)-q,z-q\rangle\leq0.
          \end{align}
Next we consider the sequences $\{a_n\}, \{\sigma_n\}$ and $\{\gamma_n\}$ in \eqref{en} as follows,
           \begin{align*}
             a_n&:=\| x_n-q\|^2\\
             \sigma_n&:=\frac{\alpha_n\| x_n-q\|^2}{2(1-\rho)}+\frac{\langle f(q)-q,x_{n+1}-q\rangle }{1-\rho}\\
             \gamma_n&:=2\beta_n^2\| e_n\|\| z_n+e_n-q\|+ 2\alpha_n\| e_n\|\| x_{n+1}-q \|\\
             +&2(1-\alpha_n-\beta_n)\| e_n-q\|\| x_{n+1}-q \|.
           \end{align*}
Clearly $\displaystyle\limsup_{n\rightarrow\infty}\sigma_n\leq0, \gamma_n\geq0$ and also $\Sigma_{n=0}^\infty\gamma_n<\infty$, hence from Lemma \ref{2.3} we have
$\displaystyle\lim_{n\rightarrow\infty}\| x_n-q\|^2=0$.\\
\textbf{Case 2.} Suppose there exists a subsequence $\{\| x_{n_j}-q\|^2\}$ of $\{\| x_n-q\|^2\}$ such that $\| x_{n_j}-q\|^2<\| x_{n_j+1}-q\|^2$ for all $j\in\mathbb{N}$. Now from Lemma \ref{2.4} there exists a nondecreasing sequence $\{m_k\}$ of $\mathbb{N}$ such that $\displaystyle\lim_{k\rightarrow\infty}m_k=\infty$ and the following inequalities hold for all $k\in\mathbb{N}$:
$$\| x_{m_k}-q\|^2\leq\| x_{m_k+1}-q\|^2\,\,\,\hbox{ and  }\,\,\,\| x_k-q\|^2\leq\| x_{m_k+1}-q\|^2.$$
Now from \eqref{xn-yn1}  we have
          \begin{align*}
            (1-\alpha_{m_k})(1-(\lambda L)^2)\| x_{m_k}-y_{m_k}\|^2\leq&\| x_{m_k}-p\|^2 -\| x_{{m_k}+1}-p\|^2\\
            +&\alpha_{m_k}\| f(x_{m_k})-p\|^2\\
            +&(1-\alpha_{m_k}-\beta_{m_k})\| p\|^2 + \| e_{m_k} \|^2 \\
            +&2 \|\alpha_{m_k}f(x_{m_k}) +\beta_{m_k}z_{m_k}-p \|\| e_{m_k}\|.
          \end{align*}
Hence, $\displaystyle\lim_{k\rightarrow\infty}\| x_{m_k}-y_{m_k}\|=0$,
therefore from \eqref{xn+1-xn} (adjusted)  $\displaystyle\lim_{k\rightarrow\infty}\| x_{{m_k}+1}-x_{m_k}\|=0$.

From \eqref{en} (adjusted) we conclude that
           \begin{align*}
             \| x_{m_k+1}-q\|^2\leq&(1-2\alpha_{m_k}(1-\rho))\| x_{m_k}-q\|^2\\
             +&2\alpha_{m_k}(1-\rho)\bigg[\frac{\alpha_{m_k}\| x_{m_k}-q\|^2}{2(1-\rho)}
             +\frac{\langle f(q)-q,x_{{m_k}+1}-q\rangle}{1-\rho}\bigg]\\
             +&2\beta_{m_k}^2\| e_{m_k}\|\| z_{m_k}+e_{m_k}-q\|+ 2\alpha_{m_k}\| e_{m_k}\|\| x_{{m_k}+1}-q \|\\
             +&2(1-\alpha_{m_k}-\beta_{m_k})\| e_{m_k}-q\|\| x_{{m_k}+1}-q \|\\
             \leq& (1-2\alpha_{m_k}(1-\rho))\| x_{m_k+1}-q\|^2\\
             +&2\alpha_{m_k}(1-\rho)\bigg[\frac{\alpha_{m_k}\| x_{m_k}-q\|^2}{2(1-\rho)}
             +\frac{\langle f(q)-q,x_{{m_k}+1}-q\rangle}{1-\rho}\bigg]\\
             +&2\beta_{m_k}^2\| e_{m_k}\|\| z_{m_k}+e_{m_k}-q\|+ 2\alpha_{m_k}\| e_{m_k}\|\| x_{{m_k}+1}-q \|\\
             +&2(1-\alpha_{m_k}-\beta_{m_k})\| e_{m_k}-q\|\| x_{{m_k}+1}-q \|.
          \end{align*}
Therefore we have
           \begin{align*}
           \| x_k-q\|^2\leq\| x_{m_k+1}-q\|^2\leq&\frac{\alpha_{m_k}\| x_{m_k}-q\|^2}{2(1-\rho)}
             +\frac{\langle f(q)-q,x_{{m_k}+1}-q\rangle}{1-\rho}\\
             +&\frac{\| e_{m_k}\|}{\alpha_{m_k}}M_1+ \frac{(1-\alpha_{m_k}-\beta_{m_k})}{\alpha_{m_k}}M_2,
           \end{align*}
where
       \begin{align*}
         M_1&=\frac{\beta_{m_k}^2}{1-\rho}\textbf{(}\| z_{m_k}+e_{m_k}-q\|+ 2\alpha_{m_k}\| x_{{m_k}+1}-q\|\textbf{)},\\
         M_2&=\frac{1}{1-\rho}\| e_{m_k}-q\|\| x_{{m_k}+1}-q \|.
       \end{align*}  From our assumptions   $\displaystyle\lim_{k\rightarrow\infty}\alpha_{m_k}=0$, $\displaystyle\lim_{k\rightarrow\infty}\frac{(1-\alpha_{m_k}-\beta_{m_k})}{\alpha_{m_k}}=0$,   $\displaystyle\lim_{k\rightarrow\infty}\frac{\| e_{m_k}\|}{\alpha_{m_k}}=0$ and \eqref{f(q)-q} (adjusted), we conclude
$\displaystyle\limsup_{k\rightarrow\infty}\| x_k-q\|^2=0$. Hence $x_k\rightarrow q$ which completes the proof of Part (ii).
\end{proof}
\begin{op}
  Can we remove the condition $\displaystyle\liminf_{n\rightarrow\infty}\| x_n -y_n\|=0$ in (i) in Theorem \ref{main1}?
\end{op}
  \section{Numerical example}
In this section the algorithm \eqref{algo2} is illustrated with some examples.
\begin{ex}\label{exam4}
Put $\alpha_n=\frac{1}{n+1} ,\; \beta_n =1-\frac{1}{n+1},\;  e_n=0,\;  H=L_2[0, 1],\; A\equiv I,\; F(x)\equiv 1,\; \lambda=\frac{1}{2}, \; C=B(0, 1)=\{f\in L_2[0,1] : \| f\| \leq 1\}$.\\
Then from the  algorithm \eqref{algo2} we have the following sequences:
                  \begin{align}\label{ex}
                    g&_n= P_C(f_n-\lambda Af_n)=P_C(\frac{f_n}{2} ),\\
                    k&_n=g_n-\lambda(Ag_n - Af_n)=\frac{1}{2}( f_n + g_n),\nonumber\\
                    f&_{n+1}= \alpha_n F(f_n) + \beta_n k_n+e_n= \frac{1}{n+1} + \frac{n}{n+1} k_n, \nonumber \\
                 \end{align}
where  $P_C$ is as follows
                 \begin{equation}\label{pc1}
                   P_{B(z,\rho)}(x)=
                   \begin{cases}
                     x & \|x-z\|\leq \rho,\\
                     z+\frac{\rho}{\|x-z\|}(x-z) & \|x-z\|> \rho,
                   \end{cases}
                 \end{equation}
where $\rho >0$ (see \cite{ca}).
We have, 
                  \begin{align*}
                    VI(C, A)=&\{f\in C: \langle Af,  g-f \rangle \geq 0\;\;,\;\;\; \forall  g \in C\}\\
                     = &\{f\in C: \langle f,  g-f \rangle \geq 0\;\;,\;\;\; \forall  g \in C \}.\\
                  \end{align*}
Obviously $0 \in VI(C, A)$, hence $ VI(C, A)\neq \emptyset$.\\
Now with $f_1=2$ and using MATLAB software we see (Figure 1) that $\{ f_n \}$ is convergence to $0$.
\end{ex}
 In the following  example, using MATLBA software, we compare  some similar algorithms and their convergence speed and behavior.
 In particular the TEGM algorithm \eqref{algo1}, VSEGM algorithm \eqref{algo3}, THEGM  algorithm \eqref{algo} and the algorithm \eqref{algo2} are compared.  We see (Figure 2)   algorithm \eqref{algo2} has a higher convergence speed than the other algorithms.
\begin{ex}\label{exam3}
  Let $\alpha_n=\frac{1}{\sqrt{n}},\; \beta_n=1-\frac{1}{\sqrt{n}}-\frac{1}{n^2},\;  \lambda=\frac{1}{2},\; e_n=0,\;  H=\mathbb{R},\; A\equiv I,\; f(x)=\frac{x}{2},\; and\; C=[0, \infty)$, $x_1=2$.
\end{ex}

Now we examine the convergence of the sequences $\{x_n\}, \{y_n\} \; and \; \{z_n\} $  in Theorem \ref{main1} in the following example.
\begin{ex}\label{exam2}
Put $\alpha_n=\frac{1}{\sqrt{n}},\; \beta_n=1-\frac{1}{\sqrt{n}}-\frac{1}{n^2},\;  \lambda=\frac{1}{2},\; e_n=\frac{1}{n^2},\;  H=\mathbb{R},\; A\equiv I,\; f\equiv 1,\; and\; C=[0, \infty)$, $x_1=2$.\\
Then we have
                  \begin{align*}
                    y&_n= P_C(x_n-\lambda Ax_n)=P_C(\frac{1}{2}x_n),\\
                    z&_n=y_n-\lambda(Ay_n - Ax_n)=\frac{1}{2}(y_n + x_n),\\
                    x&_{n+1}= \alpha_n f(x_n) + \beta_n z_n+e_n=\frac{1}{\sqrt{n}}+ (1-\frac{1}{\sqrt{n}}-\frac{1}{n^2})z_n + \frac{1}{n^2}.
                 \end{align*}
Hence,
                  \begin{align*}
                    VI(C, A)=&\{x\in C: \langle Ax,  y-x \rangle \geq 0\;\;,\;\;\; \forall  y \in C\}\\
                     = &\{x\in [0, \infty ): \langle x,  y-x \rangle \geq 0\;\;,\;\;\; \forall  y \in C \}= \{ 0 \}.
                  \end{align*}
Therefore $q=P_{VI(A, C)}\circ f(q)=P_{\{0\}}\circ(\frac{q}{2})= 0$, and
now by Theorem \ref{main1} the sequence $\{ x_n \}$ converges strongly to 0.
\end{ex}
In the following example we examine the case that  $VI(C, A)= \emptyset$ and the sequence generated by the algorithm \eqref{algo2} is divergent.
\begin{ex}\label{exam1}
Put $\alpha_n=\frac{1}{n},\; \beta_n=1-\frac{1}{n},\;  e_n=0,\;  H=\mathbb{R},\; A\equiv -1,\; f\equiv1,\; and\; C=[0, \infty)$, $x_1=2$.\\
Then we have
                  \begin{align*}
                    y&_n= P_C(x_n-\lambda Ax_n)=P_C(x_n+\frac{1}{2}),\\
                    z&_n=y_n-\lambda(Ay_n - Ax_n)=y_n,\\
                    x&_{n+1}= \alpha_n f(x_n) + \beta_n z_n+e_n=\frac{1}{n}+ (1-\frac{1}{n})P_C(x_n+\frac{1}{2}).
                 \end{align*}
Hence,
                  \begin{align*}
                    VI(C, A)=&\{x\in C: \langle Ax,  y-x \rangle \geq 0\;\;,\;\;\; \forall  y \in C\}\\
                     = &\{x\in [0, \infty ): \langle -1,  y-x \rangle \geq 0\;\;,\;\;\; \forall  y \in C \}= \emptyset.
                  \end{align*}
Now,  we will prove by induction that, for all $n\in\mathbb{N}$
                  \begin{align*}
                    y_n=P_C(x_n+\frac{1}{2})=x_n+\frac{1}{2}.
                  \end{align*}
When n=1 then we have $x_1=2\geq-\frac{1}{2}$.\\
Induction step: Suppose for n=k the inequality $x_k\geq-\frac{1}{2}$  holds.\\
Then,        \begin{align*}
                      x_{k+1}=&\frac{1}{k}+(1-\frac{1}{k})P_C(x_k+\frac{1}{2})\\
                      =& \frac{1}{k}+ (1-\frac{1}{k})(x_k+\frac{1}{2})\\
                      =& \frac{1}{k}+ (1-\frac{1}{k})x_k+ (1-\frac{1}{k})\frac{1}{2}\\
                       \geq& \frac{1}{k}+ (1-\frac{1}{k})(-\frac{1}{2})+ (1-\frac{1}{k})\frac{1}{2}\\
                       \geq& -\frac{1}{2}.
                   \end{align*}
Hence, $y_n=x_n+\frac{1}{2}$. Consequently $\displaystyle\liminf_{n\rightarrow \infty}\| x_n - y_n \| =\frac{1}{2}\neq 0$.\\
Next, we show that $\{x_n\}$ is an unbounded sequence.  Note $x_3=\frac{5}{4}>\frac{3}{4}$. \\
Induction step: Suppose for n=k the inequality $x_k > \frac{k}{4}$ holds.\\
Then,
                   \begin{align}\label{induc}
                     x_{k+1}= & \frac{1}{k}+ (1-\frac{1}{k})(x_k+\frac{1}{2})\nonumber\\
                     >& \frac{1}{k}+ (1-\frac{1}{k})(\frac{k}{4}+\frac{1}{2})\nonumber\\
                     =& \frac{1}{k}+\frac{k}{4}-\frac{1}{4}+\frac{1}{2}-\frac{1}{2k}\nonumber\\
                     =& \frac{1}{2k}+ \frac{k}{4} + \frac{1}{4}\nonumber\\
                     >& \frac{k+1}{4},
                   \end{align}
thus $\{x_n\}$ is an unbounded sequence.

\end{ex}


\end{document}